\documentclass[12pt, twoside, leqno]{article}
\usepackage{amsthm}
\usepackage{amssymb,latexsym}
\usepackage{enumerate}
\usepackage{amsfonts,amsmath,mathrsfs}
\usepackage{topcapt}
\usepackage[cp1250]{inputenc}
\usepackage[T1]{fontenc}
\usepackage{tabularx}
\usepackage{multirow}
\usepackage{epsfig}
\usepackage[figuresright]{rotating}
\usepackage{hyperref}
\hypersetup{hypertexnames=false}
\let\originalforall=\forall
\renewcommand{\forall}{\mathop{\vcenter{\hbox{\Large$\originalforall$}}}}

\let\originalexists=\exists
\renewcommand{\exists}{\mathop{\vcenter{\hbox{\Large$\originalexists$}}}}
\newtheorem{thm}{Theorem}[]

\newtheorem{lem}[thm]{Lemma}
\newtheorem{prop}[thm]{Proposition}

\newtheorem{de}[thm]{Definiton}
\begin{document}

\baselineskip=17pt

\title{The uniform invariant approximation property for compact groups}

\author{Przemysław Ohrysko\thanks{The research of this author has been supported by National Science Centre, Poland grant no. 2014/15/N/ST1/02124}\\
Institute of Mathematics\\
Polish Academy of Sciences\\
Śniadeckich 8\\
00-656 Warszawa, Poland\\
E-mail: pohrysko@impan.com\\}

\maketitle

\renewcommand{\thefootnote}{}



\renewcommand{\thefootnote}{\arabic{footnote}}
\setcounter{footnote}{0}

\begin{abstract}
In this short note we give a proof of the refined version of the uniform invariant approximation property for compact (non-commutative) groups following the Bourgain's approach (\cite{b}).
\end{abstract}
\section{Introduction}
We shall use the following notation: $G$ will stand for a compact group with the normalized Haar measure $m$ and the dual object $\Sigma$ (consisting of equivalence classes of continuous irreducible unitary representations), $L^{p}(G)$ are the usual Banach spaces of $p$-integrable functions with respect to $m$ and $M(G)$ is the convolution algebra of all complex-valued Borel regular measures endowed with the total variation norm. For $f\in L^{1}(G)$ we write $\widehat{f}(\sigma)$, $\sigma\in\Sigma$ for a matrix defined as follows
\begin{gather*}
\widehat{f}(\sigma)=\int_{G}\sigma(x^{-1})f(x)dm(x)=\\
=\int_{G}U_{x^{-1}}^{(\sigma)}f(x)dm(x)\text{ where $\sigma(x)=U_{x}^{(\sigma)}$ for all $x\in G$}.
\end{gather*}
For every $\sigma\in\Sigma$ let $d_{\sigma}$ denote the dimension (necessarily finite) of the Hilbert space $H_{\sigma}$ on which $\sigma$ acts and let $\zeta_{1}^{(\sigma)},\ldots,\zeta_{d_{\sigma}}^{(\sigma)}$ be a fixed orthonormal basis of $H_{\sigma}$. With $\sigma\in\Sigma$ and $j,k\in\{1,\ldots,d_{\sigma}\}$ we associate a coordinate function (coefficient of the representation) defined by the formula:
\begin{equation*}
u_{jk}(x)=<U_{x}^{(\sigma)}\zeta_{j}^{(\sigma)},\zeta_{k}^{(\sigma)}>.
\end{equation*}
For $y\in G$ we write $(l_{y}f)(x)=f(y^{-1}x)$ and $(r_{y}f)(x)=f(xy)$ for $f\in L^{1}(G)$, $x\in G$. A linear operator $T:X\rightarrow X$ where $X=C(G)$ or $X=L^{p}(G)$, $1\leq p<\infty$ is called \textit{invariant} if for every $y,z\in G$ we have $r_{y}l_{z}T=Tr_{y}l_{z}$.
Our main reference for harmonic analysis on compact groups is the first chapter of \cite{hr}.
\\
The uniform invariant approximation property for a wide class (\textit{translation invariant regular Banach spaces} in the terminology from \cite{k}, the prototypical examples are $L^{p}(G)$ spaces for $1\leq p<\infty$) of function spaces on a compact group $G$ is equivalent to the following theorem (see \cite{k} for details).
\begin{thm}
For every $k>1$ there exists a positive sequence $q_{k}(r)$ such that for every finite set $R\subset\Sigma$ there exists a central function $g\in L^{1}(G)$ such that:
\begin{enumerate}
  \item $\widehat{g}(\sigma)=I_{d_{\sigma}}$ for $\sigma\in R$,
  \item $\|g\|_{1}\leq k$,
  \item $v(\mathrm{supp}\widehat{g})\leq q_{k}(v(R))$ where for $S\subset\Sigma$ we put $v(S):=\sum_{\sigma\in S}d_{\sigma}^{2}$.
\end{enumerate}
\end{thm}
The most important question is how $q_{k}(v(R))$ grows with $v(R)$. It was proved in \cite{bp} that for Abelian groups one can take $q_{k}(r)\simeq r^{4r}$, later the estimate was refined (again for commutative groups) by J. Bourgain in \cite{b} to $q_{k}(r)\simeq c^{2r}$ where $c>0$ is an absolute constant. For non-Abelian groups it was proved by J. Krawczyk \cite{k} that the estimate given by Bożejko and Pełczyński holds true. In what follows we will prove that the refined estimate by J. Bourgain is correct also for non-commutative groups by extending the proof presented in \cite{w} to this setting. To be more precise our aim is to prove the following theorem.
\begin{thm}\label{glo}
Let $R\subset\Sigma$ be a finite set. Then for every $\varepsilon>0$ there exists a central function $f\in L^{\infty}(G)$ such that:
\begin{enumerate}
  \item $\widehat{f}(\sigma)=Id_{d_{\sigma}}$ for $\sigma\in R$,
  \item $\|f\|_{1}\leq 1+\varepsilon$,
  \item $v(\mathrm{supp}\widehat{f})\leq \left(\frac{c}{\varepsilon}\right)^{2v(R)}$ where $c>0$ is an absolute constant.
\end{enumerate}
\end{thm}
\section{Main result}
We need to recall first a few facts from the theory of Banach spaces. We start with II.E.13 from \cite{w}.
\begin{prop}\label{loc}
For every $n$-dimensional (complex) Banach space $X$ and for every $\delta>0$ there exists $N\leq\left(\frac{1+\delta}{\delta}\right)^{2n}$ and an embedding $u:X\rightarrow l_{\infty}^{N}$ with $(1-\delta)\|x\|\leq\|u(x)\|\leq \|x\|$.
\end{prop}
The next is III.E.14 from \cite{w}.
\begin{prop}\label{ex}
For any $\delta>0$ and every Banach space $X$, every subspace $Y\subset X$ and every finite rank operator $T:Y\rightarrow C(K)$ there exists an operator $\widetilde{T}:X\rightarrow C(K)$ such that $\widetilde{T}|_{Y}=T$ and $\|\widetilde{T}\|\leq (1+\delta)\|T\|$.
\end{prop}
\begin{de}\label{abs}
An operator $T:X\rightarrow Y$ is absolutely summing, if there exists a constant $C<\infty$ such that for all finite sequences $(x_{j})_{j=1}^{n}\subset X$ we have
\begin{equation*}
\sum_{j=1}^{n}\|Tx_{j}\|\leq C\sup\left\{\sum_{j=1}^{n}|x^{\ast}(x_{j})|:x^{\ast}\in X^{\ast},\text{ }\|x^{\ast}\|\leq 1\right\}.
\end{equation*}
We define the absolutely summing norm of an operator $T$ by
\begin{equation*}
\pi_{1}(T)=\inf\{C:\text{ the above holds for all }(x_{j})_{j=1}^{n}\subset X, \text{ }n=1,2,\ldots\}.
\end{equation*}
\end{de}
The collection of all absolutely summing operators forms an operator ideal (for a precise definition see \cite{w}). In particular, every finite rank operator is absolutely summing and $\pi_{1}(BTA)\leq \|B\|\pi_{1}(T)\|A\|$ for bounded operators $A,B$ and $T$ whenever the composition makes sense.
\begin{de}
Let $G$ be a compact group. A measure $\mu\in M(G)$ is called central if $\mu\ast\nu=\nu\ast\mu$ for every $\nu\in M(G)$, i.e. $\mu$ is in the center of the convolutive algebra $M(G)$.
\end{de}
The next theorem gives equivalent conditions for centrality (see Theorem 28.48 in \cite{hr}).
\begin{thm}\label{centr}
Let $G$ be a compact group. The following properties of a measure $\mu\in M(G)$ are equivalent:
\begin{enumerate}
  \item $\mu$ is central,
  \item $\mu\ast u_{jk}^{(\sigma)}=u_{jk}^{(\sigma)}\ast\mu$ for some set of coordinate functions $\{u_{jk}^{(\sigma)}\}$ and every $\sigma\in\Sigma$,
  \item $\widehat{\mu}(\sigma)=\alpha(\mu,\sigma)I_{d_{\sigma}}$ for all $\sigma\in\Sigma$ where $\alpha(\mu,\sigma)\in\mathbb{C}$.
\end{enumerate}
\end{thm}
Now we have a non-commutative analogue of III.F.12 from \cite{w}
\begin{prop}\label{niez}
Let $G$ be a compact group and let $T:C(G)\rightarrow C(G)$ be a bounded linear invariant operator which is absolutely summing. Then there exists a central $h\in L^{\infty}(G)$ such that $Tf=f\ast h$ for $f\in C(G)$. Moreover $\pi_{1}(T)=\|h\|_{\infty}$.
\end{prop}
\begin{proof}
By Theorem 1.2 in \cite{be} there exist $\mu,\nu\in M(G)$ such that $Tf=f\ast\mu=\nu\ast f$. Taking into account that the adjoint $T^{\ast}:M(G)\rightarrow M(G)$ is given by the very similar formula to $T$ and inserting $\delta_{e}$ into the definition of $T^{\ast}$ we obtain $\mu=\nu$. It follows now from Theorem \ref{centr} that $\mu$ is a central measure (as the coordinate functions are continuous). The rest of the proof is the same as the argument for justyfying III.F.12 in \cite{w}.
\end{proof}
We shall also use the basic Peter-Weyl theorem (see 27.40 and 28.43 in \cite{hr}).
\begin{thm}[Peter-Weyl]
Let $G$ be a compact group. The set of functions $d_{\sigma}^{\frac{1}{2}}u_{jk}^{(\sigma)}$ is an orthonormal basis for $L^{2}(G)$. Thus for $f\in L^{2}(G)$ we have
\begin{equation*}
f=\sum_{\sigma\in\Sigma}d_{\sigma}\sum_{j,k=1}^{d_{\sigma}}<\widehat{f}(\sigma)\zeta_{k}^{(\sigma)},\zeta_{j}^{(\sigma)}>u_{jk}^{(\sigma)}\text{ the series converging in $L^{2}(G)$}.
\end{equation*}
Moreover, $\|f\|^{2}=\|\widehat{f}\|^{2}=\sum_{\sigma\in\Sigma}d_{\sigma}\|\widehat{f}(\sigma)\|^{2}
_{HS}$ where
\begin{equation*}
\|\widehat{f}(\sigma)\|_{HS}=\sqrt{\mathrm{tr}(\widehat{f}(\sigma)\widehat{f}(\sigma)^{\ast})}\text{ is the Hilbert-Schmidt norm of a matrix }\widehat{f}(\sigma).
\end{equation*}
\end{thm}
\begin{lem}\label{try}
Let $G$ be a compact group and let $\sigma\in\Sigma$ and $f\in L^{1}(G)$. Then the following holds true:
\begin{enumerate}
  \item For every $y,z\in G$ and $j,k\in\{1,\ldots,d_{\sigma}\}$ we have
  \begin{equation*}
  l_{y}r_{z}u^{(\sigma)}_{jk}\in\mathrm{lin}\left\{u^{(\sigma)}_{jk}:j,k\in\{1,\ldots,d_{\sigma}\}\right\}.
  \end{equation*}
  \item If $f\ast u_{jk}^{(\sigma)}=u_{jk}^{(\sigma)}$ for every $j,k\in\{1,\ldots,d_{\sigma}\}$ then $\widehat{f}(\sigma)=I_{d_{\sigma}}$.
\end{enumerate}
\end{lem}
\begin{proof}
We have
\begin{gather*}
(l_{y}r_{z}u^{(\sigma)}_{jk})(x)= u_{jk}^{(\sigma)}(y^{-1}xz)=<U_{y^{-1}xz}^{(\sigma)}\zeta_{k}^{(\sigma)},\zeta_{j}^{(\sigma)}>=\\
=<U_{y^{-1}}^{(\sigma)}U_{x}^{(\sigma)}U_{z}^{(\sigma)}\zeta_{k}^{(\sigma)},\zeta_{j}^{(\sigma)}>=<U_{x}^{(\sigma)}U_{z}^{(\sigma)}\zeta_{k}^{(\sigma)},U_{y}\zeta_{j}^{(\sigma)}>.
\end{gather*}
Writing $U_{z}^{(\sigma)}\zeta_{k}^{(\sigma)}=\sum_{l=1}^{d_{\sigma}}c_{l}\zeta_{l}^{(\sigma)}$ and $U_{y}\zeta_{j}^{(\sigma)}=\sum_{l=1}^{d_{\sigma}}c_{l}'\zeta_{l}^{(\sigma)}$ for some complex coefficients $c_{l}$ and $c_{l}'$ we obtain the assertion of the first part of the lemma.
\\
In order to prove the second part let us observe that $\widehat{u_{jk}^{(\sigma)}}(\sigma)=e_{jk}$ (matrix unit in $M_{d_{\sigma}}(\mathbb{C})$). Hence $\widehat{f}(\sigma)e_{jk}=e_{jk}$ for every $j,k\in\{1,\ldots,d_{\sigma}\}$ which implies the desired conclusion.
\end{proof}
After these preparations we are ready to prove Theorem \ref{glo}.
\\
By Proposition \ref{niez} and the second part of Lemma \ref{try} the assertion of the theorem is equivalent to the existence of a certain linear bounded invariant operator $T:C(G)\rightarrow C(G)$.
Let us fix a number $\delta$ satisfying $0<\delta<1$. By Proposition \ref{loc} there exists a positive integer $N<\left(\frac{1+\delta}{\delta}\right)^{2v(R)}$ (observe that $v(R)=\mathrm{dim}R$) and an embedding $u:R\rightarrow l_{\infty}^{N}$ with $\|u\|\cdot \|u^{-1}\|\leq \frac{1}{1-\delta}$. Applying the Hahn-Banach theorem coordinatewise we get $\widetilde{u}:C(G)\rightarrow l_{\infty}^{N}$ - the extension of $u$ with $\|u\|=\|\widetilde{u}\|$. In addition, let $v:l^{N}_{\infty}\rightarrow C(G)$ be an extension of $u^{-1}$ with $\|v\|\leq (1+\delta)\|u^{-1}\|$ (such extension is possible by Proposition \ref{ex}). Put $T_{1}:=v\widetilde{u}:C(G)\rightarrow C(G)$. Then, obviously $T_{1}|_{R}=Id_{R}$ and using the ideal property of absolutely summing operators (see the comment following Definition \ref{abs}) and an elementary calculation $\pi_{1}\left(id:l_{\infty}^{N}\rightarrow l_{\infty}^{N}\right)=N$ we get
\begin{equation*}
\pi_{1}(T_{1})\leq \|v\|\cdot \|\widetilde{u}\|\pi_{1}\left(id:l_{\infty}^{N}\rightarrow l_{\infty}^{N}\right)\leq \|v\|\cdot\|\widetilde{u}\|\cdot N\leq N\frac{1+\delta}{1-\delta}.
\end{equation*}
We define
\begin{equation*}
T_{2}=\int_{G\times G}l_{y^{-1}}r_{z^{-1}}T_{1}r_{z}l_{y}dm(z)dm(y).
\end{equation*}
The operator $T_{2}$ is invariant and by the first part of Lemma \ref{try} we have $T_{2}|_{R}=Id$. Moreover,
\begin{gather*}
\|T_{2}\|\leq \|T_{1}\|\leq \frac{1+\delta}{1-\delta},\\
\pi_{1}(T_{2})\leq \pi_{1}(T_{1})\leq N\cdot \frac{1+\delta}{1-\delta}.
\end{gather*}
From Proposition \ref{niez} (actually, we use the version of Proposition \ref{niez} for functions which is explicitly stated as Theorem 28.49 in \cite{hr}) we infer that $T_{2}$ is a convolution with
a central $h\in L^{\infty}(G)$ satisfying
\begin{gather*}
\widehat{h}(\sigma)=Id_{d_{\sigma}} \text{ for } \sigma\in R,\\
\|h\|_{1}\leq  \frac{1+\delta}{1-\delta},\\
\|h\|_{\infty}\leq N \frac{1+\delta}{1-\delta}.
\end{gather*}
Last two inequalities give $\|h\|_{2}\leq \frac{1+\delta}{1-\delta} \sqrt{N}$. Let us define $g=h\ast h\ast h$. Then $g$ is also central and by Theorem \ref{centr} we have  $\widehat{g}(\sigma)=\alpha(g,\sigma)Id_{d_{\sigma}}=\alpha^{3}(h,\sigma)Id_{d_{\sigma}}$ for every $\sigma\in\Sigma$. Applying the Peter-Weyl theorem to $g$ we have
\begin{equation*}
g=\sum_{\sigma\in\Sigma}d_{\sigma}\alpha(g,\sigma)\sum_{j=1}^{d_{\sigma}}u_{jj}^{(\sigma)}.
\end{equation*}
Put
\begin{equation*}
f=\sum_{\sigma\in\Sigma:|\alpha(g,\sigma)|>N^{-4}}d_{\sigma}\alpha(g,\sigma)\sum_{j=1}^{d_{\sigma}}u_{jj}^{(\sigma)}.
\end{equation*}
Then, using the equality $\|h\|_{2}=\|\widehat{h}\|_{HS}$, we obtain
\begin{gather*}
\|f\|_{1}\leq \|g\|_{1}+\|g-f\|_{1}\leq \|h\|_{1}^{3}+\sum_{\sigma\in\Sigma:|\alpha(g,\sigma)|\leq N^{-4}}d_{\sigma}^{2}|\alpha(g,\sigma)|\leq\\
\leq \left(\frac{1+\delta}{1-\delta}\right)^{3}+N^{-\frac{4}{3}}\sum_{\sigma\in\Sigma}d_{\sigma}^{2}|\alpha(h,\sigma)|^{2}=\\
=\left(\frac{1+\delta}{1-\delta}\right)^{3}+N^{-\frac{4}{3}}\|h\|_{2}^{2}\leq \left(\frac{1+\delta}{1-\delta}\right)^{3}+\left(\frac{1+\delta}{1-\delta}\right)^{2}N^{-\frac{1}{3}}.
\end{gather*}
Finally, with the aid of $L^{2}$ theory again, we get
\begin{gather*}
v\left(\{\sigma\in\Sigma:\widehat{f}(\sigma)\neq 0\}\right)\leq\\
\leq\sum_{\sigma\in\Sigma:|\alpha(h,\sigma)|>N^{-\frac{4}{3}}}d_{\sigma}^{2}\leq N^{\frac{8}{3}}\sum_{\sigma\in\Sigma:|\alpha(h,\sigma)|>N^{-\frac{4}{3}}}|\alpha(h,\sigma)|^{2}d_{\sigma}^{2}\leq\\
\leq N^{\frac{8}{3}}\|h\|^{2}_{2}\leq N^{\frac{11}{3}}\left(\frac{1+\delta}{1-\delta}\right)^{2}\leq N^{4}\left(\frac{1+\delta}{1-\delta}\right)^{2}.
\end{gather*}
Chosing correct $\delta$ to $\varepsilon$ finishes the proof (the exact dependence is difficult to calculate but asymptotically $\varepsilon\simeq\delta^{4}$).

\end{document}